\documentclass[reqno,11pt]{amsart}
\usepackage{amsmath,amsrefs,amssymb,color}
\numberwithin{equation}{section}

\DeclareMathOperator{\divergence}{div}

\newcommand{\R}{\mathbb{R}}
\renewcommand{\S}{\mathbb{S}}
\newcommand{\N}{\mathbb{N}}

\newcommand{\<}{\left<}
\renewcommand{\>}{\right>}
\renewcommand{\[}{\left[}
\renewcommand{\]}{\right]}
\renewcommand{\(}{\left(}
\renewcommand{\)}{\right)}

\newtheorem{theorem}{Theorem}[section]
\newtheorem{proposition}[theorem]{Proposition}

\begin{document}

\title[The sixth-order $Q$-curvature of conformal metrics in $\R^n$]{Positivity and non-positivity results for the sixth-order $Q$-curvature of conformal metrics in $\R^n$}

\author{J\'er\^ome V\'etois}

\address{J\'er\^ome V\'etois, Department of Mathematics and Statistics, McGill University, 805 Sherbrooke Street West, Montreal, Quebec H3A 0B9, Canada}
\email{jerome.vetois@mcgill.ca}

\author{Samuel Zeitler}

\address{Samuel Zeitler, Department of Mathematical Sciences, Carnegie Mellon University, Pittsburgh, Pennsylvania 15213, USA}

\email{szeitler@andrew.cmu.edu}

\thanks{The first author was supported by the NSERC Discovery Grant RGPIN-2022-04213.}

\date{July 17, 2026}

\begin{abstract}
Given $n,m\in\N$ such that $n\ge2m\ge4$, letting $g$ be a conformally Euclidean metric on $\R^n$, we consider the question of positivity of the lower-order $Q$-curvatures $Q_g^{\(2k\)}$ for $k\in\left\{1,\dotsc,m-1\right\}$ when $Q_g^{\(2m\)}$ is assumed to be nonnegative and not identically zero. We assume moreover that the scalar curvature of the metric $g$ is nonnegative near infinity if $n=2m$ or that $Q_g^{\(2m\)}$ satisfies a slow decay barrier condition near infinity if $n>2m$. Positive results for this question have been obtained by Gursky and Malchiodi~\cite{GM} for $m=2$ and $k=1$ in the context of closed manifolds with nonnegative scalar curvature and by Li and Xu~\cite{LX2} {and Li, Wei, and Xu~\cite{LWX}} for $m\ge2$ and $k\in\left\{1,\min\(m-1,2\)\right\}$ in the context of conformally Euclidean metrics on $\R^n$. {These results hold for all $n\ge2m$.} Considering the case where $m\ge4$ and $k=3$, we obtain a positive result for this question when $n\in\left\{2m,2m+1,\dotsc,4m-6\right\}${, namely for these dimensions, we obtain that if $Q_g^{\(2m\)}\ge0$ and $Q_g^{\(2m\)}\not\equiv0$ in $\R^n$, then $Q_g^{\(6\)}>0$}. On the other hand, in surprising contrast with the results of~{\cites{GM,LX2,LWX}}, we find that the answer {to this question} is negative when $k=3$ and $n\ge N_m$ for some $N_m\in\R$. In this case, we are able to construct examples of conformally Euclidean metrics such that $Q_g^{\(2m\)}$ is positive everywhere, but $Q_g^{\(6\)}$ is negative at some point. By stereographic projection, our examples extend to metrics conformal to the standard metric on $\S^n$.
\end{abstract}

\maketitle

\section{Introduction and main result}\label{Sec1}

Let $\varphi$ be a smooth and positive function {on $\R^n$} and $g=\varphi^2\left|dx\right|^2$ be a conformally Euclidean metric on $\R^n$, $n\ge2$. For every $m\in\N$ such that $n \geq 2m$, the $Q$-curvature of order $2m$ of the metric $g$ is given by 
\begin{equation}\label{Sec1Eq1}
Q_g^{\(2m\)}=\left\{\begin{aligned}&\frac{2}{n-2m}\,\varphi^{-\frac{n+2m}{2}}\(-\Delta\)^m\Big(\varphi^{\frac{n-2m}{2}}\Big)&&\text{if }n>2m\\&\varphi^{-n}\(-\Delta\)^m\(\ln\varphi\)&&\text{if }n=2m.\end{aligned}\right.
\end{equation}
These curvatures were introduced by Branson and \O rsted~\cites{BO} for $n=2m=4$ and extended by Branson~\cites{B} in general dimensions. They constitute a widely studied family of conformal invariants which naturally extend the scalar curvature to higher orders. Their construction is based on Fefferman and Graham's ambient metric~\cites{FG1,FG2}, as well as on Graham, Jenne, Mason, and Sparling's GJMS operators~\cite{GJMS}. 

\smallskip
The case where $Q_g^{\(2m\)}\equiv1$ has received special attention. In the context we are considering, this case is well understood: up to a finite total $Q$-curvature condition in the case where $n=2m$, the conformally Euclidean metrics solving this equation are known to be radially symmetric and have been fully classified. On this subject, we refer to the historical articles by Obata~\cite{O} and Caffarelli, Gidas, and Spruck~\cite{CGS} for $m=1$ and to the later articles by Lin~\cite{Lin}, Wei and Xu~\cite{WX}, and Martinazzi~\cite{M} for higher orders.

\smallskip
In this article, we are interested in the question of positivity of the lower-order $Q$-curvatures $Q_g^{\(2k\)}$ for $k\in\left\{1,\dotsc,m-1\right\}$ when $Q_g^{\(2m\)}$ is assumed to be nonnegative. In the case where $m=2$ and $k=1$, this question was answered positively by Gursky and Malchiodi~\cite{GM} in the context of a closed manifold with non-negative scalar curvature (see also in this case the article by Fazly, Wei, and Xu~\cite{FWX} where this result was obtained for a class of $Q$-curvature-type equations in $\R^n$). This result has proven to be useful in establishing several fundamental results. {A brief discussion of these and related results is presented at the end of the introduction.}

\smallskip
Gursky and Malchiodi's curvature positivity result was recently shown to remain valid in the context of conformally Euclidean metrics on $\R^n$ by Li and Xu~\cite{LX2} for {$n>2m$,} $m\ge2$ and $k\in\left\{1,\min\(m-1,2\)\right\}$ and by Li, Wei, and Xu~\cite{LWX} {for $n=2m$, $m\ge2$} and $k\in\left\{1,\min\(m-1,2\)\right\}$, {under a suitable assumption which differs for $n>2m$ and $n=2m$ (see (A) below)}. This remarkable improvement naturally led the authors of these articles to conjecture that this result should remain valid in the case where $1\le k\le m-1$ and $n\ge2m$. In this article, however, we show that the situation is surprisingly different for $k=3$. In this case, our results show that the positivity results of~{\cites{GM,LX2,LWX}} remain valid when $n$ is not too large relative to $m$, but is not true for all sufficiently large $n$. 

\smallskip
Throughout our work, as in the articles by Li and Xu~\cite{LX2} and Li, Wei, and Xu~\cite{LWX}, we make the following assumption:
\begin{itemize}
\item[(A)] If $n>2m$, then we assume that there exists $s\in\(-2m,0\)$ and $c,R\in\(0,\infty\)$ such that $Q_g^{\(2m\)}\ge c\left|x\right|^s$ for all $x\in\R^n\setminus B\(0,R\)$. If $n=2m$, then we assume that there exists $R\in\(0,\infty\)$ such that  $Q_g^{\(2\)}\ge0$ and $Q_g^{\(n\)}\ge0$ in $\R^n\setminus B\(0,R\)$. If $n=2m$, then we assume moreover that $Q_g^{\(n\)}\not\equiv0$ in $\R^n$.
\end{itemize}
{The last point in (A) is not assumed in \cite{LWX}, but it allows us to rule out a trivial case. Assuming (A), we can} use an integral representation formula for the conformal factor of the metric $g$. More precisely, by letting
$$u:=\left\{\begin{aligned}&\varphi^{\frac{n-2m}{2}}&&\text{if }n>2m\\&\ln\varphi&&\text{if }n=2m\end{aligned}\right.
$$
and applying Li and Xu's Theorem~2.6~\cite{LX2} for $n>2m$ and Li, Wei, and Xu's Theorems~3.8 and~3.9~\cite{LWX} for $n=2m$, we obtain
\begin{equation}\label{Sec1Eq2}
u\(x\)=\left\{\begin{aligned}&\int_{\R^n}\left|x-y\right|^{2m-n}d\mu_g^{\(2m\)}\(y\)&&\text{if }n>2m\\&\int_{\R^n}\ln\(\frac{\left|y\right|}{\left|x-y\right|}\)d\mu_g^{\(n\)}\(y\)+C&&\text{if }n=2m\end{aligned}\right.\quad\forall x\in\R^n,
\end{equation}
where $C\in\R$ is a constant, $d\mu_g^{\(2m\)}$ is the measure defined by
$$d\mu_g^{\(2m\)}\(y\):=\frac{\Gamma\(\frac{n}{2}-m+1\)}{2^{2m-1}\(m-1\)\,!\,\Gamma\(\frac{n}{2}\)\omega_{n-1}}\,Q_g^{\(2m\)}\(y\){u\(y\)^{\frac{n + 2m}{n-2m}}} dy\quad\forall y\in\R^n$$
if $ n > 2m $ and 
$$d\mu_g^{\(n\)}\(y\):=\frac{2}{(n-1)!\,\omega_{n}}\,Q_g^{\(n\)}\(y\){e^{nu\(y\)}} dy\quad\forall y\in\R^n$$
if $n=2m$. Moreover, we can differentiate up to $2m-1$ times the integral in \eqref{Sec1Eq2} under the integral sign. In the case where $n=2m$, this formula relates to the concept of normal metrics introduced by Finn~\cite{Fi} for $m=1$, extended to higher orders by Chang, Qing, and Yang~\cite{CQY}, Fang~\cite{Fa}, and Ndiaye and Xiao~\cite{NX}, and subsequently studied in several articles including the recent articles by Li~\cites{Li1,Li2,Li3}, Li, Wei, and Xu~\cite{LWX}, and Li and Xu~\cite{LX3}.

\smallskip
We obtain the following positivity result:

\begin{theorem}\label{Th1}
For every $m\ge4$, if $n\in\left\{2m, 2m + 1,\dotsc,4m-6\right\}$ and $g$ is a conformally Euclidean metric on $\R^n$ satisfying (A) and $Q_g^{\(2m\)}\ge0$ in $\R^n$, then $Q_g^{\(6\)}>0$ in $\R^n$.
\end{theorem}

On the other hand, for large values of $n$, we obtain the following {non-positivity} result:

\begin{theorem}\label{Th2}
For every $m\in\N$ such that $m \geq 4$, let $N_m$ be the largest real root of the polynomial
\begin{align}\label{Th2Eq1}
\gamma_m\(n\):=&n^5-19\(m-1\)n^4+\(117m^2-212m+92\)n^3-\(337m^3-851m^2\right.\nonumber\\
&\left.+652m-140\)n^2+2\(m-1\)\(241m^3-518m^2+264m-24\)n\nonumber\\
&-4m\(m-1\)\(70m^3-187m^2+133m-24\).
\end{align}
Then for every $n>N_m$, there exists a conformally Euclidean metric $g$ on $\R^n$ such that $Q_g^{\(6\)}\(0\)<0$ and $C^{-1}\le Q_g^{\(2m\)}\le C$ in $\R^n$ for some constant $C>0$. Additionally, by stereographic projection, $g$ extends to a smooth metric on $\S^n$. Furthermore, $N_m\sim\Lambda m$ as $m\to\infty$ and {$N_m<\Lambda m$} for all $m\in\N$ with $ m \geq 4$, where $\Lambda\simeq10.55$ is the largest real root of the polynomial $X^4-17X^3+83X^2-171X+140$. 
\end{theorem}

For small values of $m$, numerical computations give $N_4\simeq29.1$, $N_5\simeq39.7$, $N_6\simeq50.3$, $N_7\simeq60.9$, $N_8\simeq71.4$, $N_9\simeq82$, and $N_{10}\simeq92.6$.

\smallskip
Both the proofs of Theorems~\ref{Th1} and~\ref{Th2} rely on an integral formula for $Q_g^{\(6\)}$ which we derive in Section~\ref{Sec2}. We then prove Theorem~\ref{Th1} in Section~\ref{Sec3} and Theorem~\ref{Th2} in Section~\ref{Sec4}. Some comments are in order. The proof of Theorem~\ref{Th1} reduces to studying the positivity of a cubic polynomial of multiple nonnegative variables. Note that this is clearly more difficult than in the cases of $Q_g^{\(2\)}$ and $Q_g^{\(4\)}$, where the polynomial {to study} is of degree 1 and 2, respectively. The proof is then divided into several cases. When $2m+1\le n\le \frac{5m-3}{2}$ or $n\in\left\{\frac{8m-6}{3},3m-3,4m-6\right\}$, all the coefficients of the polynomial are nonnegative. When $\frac{5m-3}{2}<n<\frac{8m-6}{3}$ or $\frac{8m-6}{3}<n<3m-3$, only one coefficient is negative. The most difficult case that we deal with in Theorem~\ref{Th1} is the case where $3m-3<n<4m-6$, in which four coefficients are negative. In the remaining case where $n>4m-6$, three coefficients are negative. While we do not think that the dimension $4m-6$ is optimal in Theorem~\ref{Th1}, it is clear from Theorem~\ref{Th2} that the result does not remain valid for all $n>4m-6$. The proof of Theorem~\ref{Th2} relies {again on the study of the cubic polynomial obtained in Section~\ref{Sec2}. For large values of $n$}, we manage to obtain our non-positivity result for $Q_g^{\(6\)}$ by {constructing a suitable family of conformal metrics with peaks of different heights at two different points}.

{\smallskip
We conclude this introduction by briefly mentioning some results in the context of closed manifolds which have been obtained thanks to Gursky and Malchiodi's curvature positivity result, as well as some related results. First, Gursky and Malchiodi~\cite{GM} used this result to obtain positivity results on the associated Paneitz--Branson operator and on its Green's function, which in turn allowed them to obtain an existence result for the constant fourth-order $Q$-curvature problem. In this regard, we also mention the article by Hang and Yang~\cite{HY1} where a different approach is used for these results and the positive scalar curvature assumption is relaxed. More generally, positivity properties of GJMS operators and of their Green's functions are known to have applications to several fundamental results including existence, quantitative stability, and compactness results of metrics with prescribed $Q$-curvature or, more generally, positive solutions to $Q$-curvature-type equations. In addition to the already mentioned articles~\cites{GM,HY1}, earlier and later references on existence results include articles by Djadli, Hebey and Ledoux~\cite{DHL}, Esposito and Robert~\cite{ER}, Robert~\cites{R1,R2}, Hang and Yang~\cites{HY2}, Gursky, Hang, and Lin~\cite{GHL}, Chen and Hou~\cite{CH}, Mazumdar and V\'etois~\cite{MV}, and Mazumdar and Ndiaye~\cite{MN}; see also the article by Hyder and Sire~\cite{HS} on the existence of singular solutions. A reference on quantitative result is by Andrade, K\"onig, Ratzkin, and Wei~\cite{AKRW}. As regards compactness results, we refer to the articles by Hebey and Robert~\cite{HR}, Hebey, Robert, and Wen~\cite{HRW}, Li and Xiong~\cite{LX1}, Gong, Kim, and Wei~\cite{GKW}, and Mazumdar and Premoselli~\cite{MP}. We now mention another type of applications of Gursky and Malchiodi's curvature positivity result. As can be seen in the articles by V\'etois~\cite{V}, Case~\cite{C}, and Li and Wei~\cite{LW}, this result has also proven to be useful in establishing uniqueness results of conformal metrics with prescribed $Q$-curvature or, again, more generally, positive solutions to $Q$-curvature-type equations. We refer to the article by Case and Gover~\cite{CG} for a recent survey on GJMS operators and $Q$-curvatures. We also mention a recent article by Ge, Wang, and Wei~\cite{GWW} where analogues of the above-mentioned existence and uniqueness results have been obtained for a problem involving the quotient between the $Q$-curvature of order 4 and the scalar curvature.}

\section{An integral formula for $Q_g^{\(6\)}$}\label{Sec2}

{Let $n,m\in\N$ be such that $n\ge 2m\ge8$ and $g$ be a conformally Euclidean metric on $\R^n$ satisfying (A).} Throughout this section and the next, {we use the notation
$$d\mu_g^{\(2m\)}\(y\):=d\mu_g^{\(2m\)}\(y_1\)\times\dotsb\times d\mu_g^{\(2m\)}\(y_l\)$$
for all $l\in\N$ and $y=\(y_1,\dotsc,y_l\)\in\R^{ln}$. We denote} $y=y_1$ if $l=1$ and $y=\(y_1,\dotsc,y_l\)$ if $l\ge2$. For conciseness, we use the notation 
$$\left|x - y\right|^{\[\theta_1, \dots, \theta_l\]}  := \prod_{i=1}^l \left|x - y_i\right|^{\theta_i}  $$ 
for all $l\in\N$, $\theta_1,\dotsc,\theta_l\in\R$, $x\in\R^n$, and $y=\(y_1,\dotsc,y_l\)\in\R^{ln}$.
We also define, for $ i, j \in \mathbb{N} $ with $ i \neq j $, $$Z_{i,j}\(x,y\) :=\frac{\left|y_i - y_j\right|^2}{\left|x - y_i\right|^{2}\left|x - y_j\right|^{2}}\quad\text{and}\quad Z_i\(x,y\) :=\frac{1}{\left|x - y_i\right|^{2}} $$ 
for all $x,y_i, y_j \in\R^n$. We note 
$$\frac{\< x - y_i, x - y_j \>}{\left|x - y_i\right|^{2}\left|x - y_j\right|^{2}}=\frac{1}{2}\(Z_i\(x,y\)+Z_j\(x,y\)-Z_{i,j}\(x,y\)\),$$
which gives (assuming we can differentiate under the integral sign)
\begin{align}
&\Delta \(\int_{\mathbb{R}^{ln}}  {\left|x - y\right|^{\[\theta_1, \dots, \theta_l\]}}f\(y\){d\mu_g^{\(2m\)}\(y\)}\)\nonumber\\
&\quad=\int_{\mathbb{R}^{ln}}\left|x - y\right|^{\[\theta_1,\dots, \theta_l\]}f\(y\)\Bigg(\(\sum_{i=1}^l\theta_i+n-2\)\(\sum_{i=1}^l\theta_iZ_i\(x,y\)\)\nonumber\\
&\qquad-\sum_{1 \leq i < j \leq l} \theta_i\theta_jZ_{i,j}\(x,y\)\Bigg){d\mu_g^{\(2m\)}\(y\)},\label{general laplacian}\\
&\Delta u\(x\)=-\(2m-2\)\max\(n-2m,1\)\int_{\mathbb{R}^n}\left|x - y\right|^{2m-n-2}{d\mu_g^{\(2m\)}\(y\)},\label{laplacian}
\end{align}
and 
\begin{align}\label{general inner product}
&\< \nabla u\(x\), \nabla \(\int_{\mathbb{R}^{ln}}  {\left|x - y\right|^{\[\theta_1, \dots,\theta_l\]}}f\(y\){d\mu_g^{\(2m\)}\(y\)} \) \>\nonumber\\
&\quad=-\frac{1}{2}\max\(n-2m,1\)\int_{\mathbb{R}^{\(l+1\)n}}\left|x - y\right|^{\[\theta_1,\dots, \theta_l , 2m-n\]}f\(y\)\sum_{i=1}^{l} \theta_i\(Z_i\(x,y\) \right.\nonumber\\
&\qquad\left. + Z_{l+1}\(x,y\)- Z_{i, l+1}\(x,y\)\){d\mu_g^{\(2m\)}\(y\)}, 
\end{align}
where $u$ is as in \eqref{Sec1Eq2}.

\smallskip
This section is devoted to proving the following result:

\begin{proposition}\label{Pr}
Let $n,m\in\N$ be such that $n\ge 2m\ge8$ and $g$ be a conformally Euclidean metric on $\R^n$ satisfying (A). If $n > 2m$, then
\begin{align}\label{formula laplacian cubed}
Q_g^{\(6\)}\(x\)&= 2\(m-3\)\int_{\mathbb{R}^{6n}} \prod_{i=1}^6 \left|x - y_i\right|^{2m - n}P_{n,m}\(Z\(x,y\)\){d\mu_g^{\(2m\)}\(y\)},
\end{align}   
where $Z:=\(Z_{i,j},Z_i\)_{1\le i < j\le 6}$ and $P_{n,m}$ is the cubic polynomial defined by
\begin{align*}
&P_{n,m}\(Z\)\\
&\quad:=Z_{1,2}\big(\(2m-n-2\)^2\(2m-n-4\)^2 Z_{1,2}^2\\
&\qquad- 6\(4m - n - 6\)\(2m-n-2\)^2\(2m-n-4\)Z_{1,2} Z_{1,3}\allowdisplaybreaks\\
&\qquad+ 3\(4m - n - 6\)\(3m - n - 3\)\(2m-n-2\)^2 Z_{1,2}Z_{3,4}\allowdisplaybreaks\\
&\qquad+4\(4m - n  - 6\)\(3m - n - 3\)\(2m - n - 2\)\(2m - n - 4\)Z_{1,3}Z_{1,4}\allowdisplaybreaks\\ 
&\qquad-2\(4m - n - 6\)\(2m - n - 2\)^3Z_{1,3}Z_{2,3}\allowdisplaybreaks\\
&\qquad+12\(4m - n - 6\)\(3m - n - 3\)\(2m-n-2\)^2Z_{1,3} Z_{3,4}\allowdisplaybreaks\\
&\qquad- 6\(4m - n - 6\)\(3m - n - 3\)\(8m-3n-6\)\(2m - n - 2\)Z_{1,3} Z_{5,6}\\ 
&\qquad+\(4m - n - 6\)\(3m - n - 3\)\(5m - 2n - 3\)\(8m-3n-6\)Z_{3,4} Z_{5,6}\big).
\end{align*}
 If $ n = 2m$, then 
\begin{align}\label{formula laplacian cubed n equals 2m}
Q_g^{\(6\)}\(x\)&= e^{-6u} \int_{\mathbb{R}^{6n}} P_{n,m}\(Z\(x,y\)\){d\mu_g^{\(n\)}\(y\)},
\end{align}    
where $Z:=\(Z_{i,j},Z_i\)_{1\le i < j\le 6}$ and $P_{n,m}$ is defined by
\begin{align*}
&P_{n,m}\(Z\) \\
 &:=  \frac{4\(n-6\)\(2-\lambda\)}{\lambda^5}\(n - 2 - \frac{n-6}{2}\lambda\)\(n - 4 - \frac{n-6}{2}\lambda\)Z_1^3\\
 &\ + \frac{6\(n-6\)^2\(2-\lambda\)}{\lambda^4}\(n - 4 - \frac{n-6}{2}\lambda\) Z_1^2 Z_{1,2}\allowdisplaybreaks\\
 &\ + \frac{3\(n-6\)^2\(2-\lambda\)}{2\lambda^4}\(n - 2 - \frac{n-6}{2}\lambda\)\(n - 4 - \frac{n-6}{2}\lambda\) Z_1^2 Z_2\allowdisplaybreaks\\
 &\ + \frac{3\(n-6\)^3\(2-\lambda\)}{8\lambda^3}\(n - 4 - \frac{n-6}{2}\lambda\) Z_1^2 Z_{2,3} 
+ \frac{12\(n-6\)^2\(2-\lambda\)}{\lambda^4}  Z_1 Z_{1,2}^2\allowdisplaybreaks\\
 &\ + \frac{3\(n-6\)^3\(2-\lambda\)}{\lambda^3} Z_1 Z_{1,2} Z_{1,3}\allowdisplaybreaks\\
 &\ + \frac{3\(n-6\)^2\(2-\lambda\)}{\lambda^4}\(n - 4 - \frac{n-6}{2}\lambda\) Z_1 Z_2 Z_{1,2}\allowdisplaybreaks\\
 &\ + \frac{3\(n-6\)^3\(2-\lambda\)}{\lambda^3} Z_1 Z_{1,2} Z_{2,3} 
+ \frac{3\(n-6\)^4\(2-\lambda\)}{8\lambda^2} Z_1 Z_{1,2} Z_{3,4}\allowdisplaybreaks\\
 &\ + \frac{3\(n-6\)^3\(2-\lambda\)}{2\lambda^3}\(n - 4 - \frac{n-6}{2}\lambda\) Z_1 Z_2 Z_{1,3}\allowdisplaybreaks\\
 &\ + \frac{\(n-6\)^3\(2-\lambda\)}{8\lambda^3}\(n - 2 - \frac{n-6}{2}\lambda\)\(n - 4 - \frac{n-6}{2}\lambda\) Z_1 Z_2 Z_3\allowdisplaybreaks\\
 &\ + \frac{3\(n-6\)^4\(2-\lambda\)}{32\lambda^2}\(n - 4 - \frac{n-6}{2}\lambda\)Z_1 Z_2 Z_{3,4} 
\allowdisplaybreaks\\
 &\ +\frac{3\(n-6\)^3\(2-\lambda\)}{4\lambda^3}  Z_1 Z_{2,3}^2+ \frac{3\(n-6\)^4\(2-\lambda\)}{8\lambda^2} Z_1 Z_{2,3} Z_{2,4}\allowdisplaybreaks\\
 &\ + \frac{3\(n-6\)^5\(2-\lambda\)}{128\lambda} Z_1 Z_{2,3} Z_{4,5} + \frac{16\(n-6\)}{\lambda^4} Z_{1,2}^3 
+ \frac{12\(n-6\)^2}{\lambda^3} Z_{1,2}^2 Z_{1,3}\allowdisplaybreaks\\
 &\ + \frac{3\(n-6\)^3}{8\lambda^2} Z_{1,2}^2 Z_{3,4} + \frac{\(n-6\)^3}{\lambda^2} Z_{1,2} Z_{1,3} Z_{1,4} + \frac{2\(n-6\)^2}{\lambda^3} Z_{1,2} Z_{1,3} Z_{2,3}\\
 &\  + \frac{3\(n-6\)^3}{2\lambda^2} Z_{1,2} Z_{1,3} Z_{3,4} + \frac{3\(n-6\)^4}{16\lambda} Z_{1,2} Z_{1,3} Z_{5,6}+ \frac{\(n-6\)^5}{256} Z_{1,2} Z_{3,4} Z_{5,6},
\end{align*}
where $ \lambda  := \mu_g^{\(2m\)}\(\R^n\)$.
\end{proposition}

We note that for $n=2m$, we have $0<\lambda\le2$ as a consequence of (A) together with Li, Wei, and Xu's Theorem~3.8~\cite{LWX}.

\proof[Proof of Proposition 2.1]
We begin with assuming $ n > 2m $. Let $\varphi$ be a smooth and positive function {on $\R^n$}, $g=\varphi^2\left|dx\right|^2$ be a conformally Euclidean metric on $\R^n$, and $u$ be as in \eqref{Sec1Eq2}. We begin by computing $ \Delta^2 \(u^t\) $, where
$$t := \frac{n-6}{n-2m}.$$ 
We note 
$$ t - 1 = \frac{2m-6}{n-2m} .$$
Simple computations show
$$\Delta\(u^t\) = \divergence\(t u^{t-1} \nabla u\) = t \(t-1\) u^{t-2}\left|\nabla u\right|^2 + t u^{t-1} \Delta u .$$
Plugging \eqref{laplacian} and \eqref{general inner product} into this formula gives
\begin{align}\label{laplacian for u^t}
u\(x\)^{2-t}\Delta\(u^t\)\(x\)&= -\(n-6\) \int_{\mathbb{R}^{2n}}\left|x - y_1\right|^{2m - n} \left|x - y_2\right|^{2m - n}\nonumber\\
&\quad{\times\(4Z_1\(x,y\)+\(m-3\)Z_{1,2}\(x,y\)\)d\mu_g^{\(2m\)}\(y\)}.
\end{align}
We now write
$$ t - 2 = \frac{ 4m -n- 6 }{n-2m}\quad\text{and}\quad t-3 = \frac{6m -2n- 6}{n-2m}.$$
Similarly as in \eqref{laplacian for u^t}, we calculate
\begin{align}\label{laplacian for u^{t-2}}
&u\(x\)^{4-t}\Delta\(u^{t-2}\)\(x\)\nonumber\\
&\quad=\(4m-n-6\)\int_{\mathbb{R}^{2n}}\left|x - y_3\right|^{2m-n}\left|x - y_4\right|^{2m-n}{\(\(4m-2n-4\)Z_3\(x,y\)\right.}\nonumber\\
&\qquad{\left.-\(3m-n-3\)Z_{3,4}\(x,y\)\)d\mu_g^{\(2m\)}\(y\)}.
\end{align}
We then have 
\begin{align}\label{laplacian squared}
 \Delta^2\(u^t\)&=\Delta\(u^{t-2}\cdot u^{2-t} \Delta\(u^t\)\)\nonumber\\
&= u^{2-t}\Delta\(u^{t-2}\)\Delta\(u^t\)+2\<\nabla\(u^{t-2}\) , \nabla\(u^{2-t}\Delta\(u^t\)\)\>\nonumber\\
 &\quad+ u^{t-2} \Delta \(u^{2-t}\Delta\(u^t\)\).
\end{align}
Applying (\ref{general laplacian}), (\ref{general inner product}), \eqref{laplacian for u^t}, \eqref{laplacian for u^{t-2}}, and (\ref{laplacian squared}) gives
\begin{align*}
&u\(x\)^{4-t}\Delta^2\(u^t\)\(x\)\nonumber\\
&\quad=-\(n-6\)\int_{\mathbb{R}^{4n}}\prod_{i=1}^4\left|x - y_i\right|^{2m-n}\big(\(4m-n-6\)\(4Z_1+\(m-3\)Z_{1,2}\)\nonumber\\
&\qquad\times\(\(4m-2n-4\)Z_3-\(3m-n-3\)Z_{3,4}\)-4\(4m-n-6\)Z_1\nonumber\\
&\qquad\times\big(\(2m-n-2\)Z_1+\(2m-n\)Z_2+\(4m-2n-2\)Z_3\nonumber\\
&\qquad-\(2m-n-2\)Z_{1,3}-\(2m-n\)Z_{2,3}\big)-\(m-3\)\(4m-n-6\)\nonumber\\
&\qquad\times\(2m-n-2\)Z_{1,2}\(Z_1+Z_2+2Z_3-Z_{1,3}-Z_{2,3}\)\nonumber\allowdisplaybreaks\\
&\qquad+4Z_1\big(\(4m-n-4\)\(\(2m-n-2\)Z_1+\(2m-n\)Z_2\)\nonumber\\
&\qquad-\(2m-n\)\(2m-n-2\)Z_{1,2}\big)+\(m-3\)\(2m-n-2\)Z_{1,2}\nonumber\\
&\qquad\times\(\(4m-n-6\)\(Z_1+Z_2\)-\(2m-n-2\)Z_{1,2}\)\big){\(x,y\)d\mu_g^{\(2m\)}\(y\)}.
\end{align*} 
By interchanging the order of integration, we then obtain
\begin{align}\label{formula for laplacian squared}
&u\(x\)^{4-t}\Delta^2\(u^t\)\(x\)\nonumber\\
&\quad= \(n-6\)\int_{\mathbb{R}^{4n}}\prod_{i=1}^4\left|x - y_i\right|^{2m-n}\big(c_1Z_1^2+ c_2 Z_1 Z_2+c_3 Z_1 Z_{1,2} + c_4 Z_1 Z_{2,3}\nonumber\\
&\qquad+c_5 Z_{1,2}^2+c_6Z_{1,2} Z_{1,3}+ c_7 Z_{1,2} Z_{3,4}\big){\(x,y\)d\mu_g^{\(2m\)}\(y\)},
\end{align}
where
\begin{align*}
c_1&:=4\(4m-n-6\)\(2m-n-2\)-4\(4m-n-4\)\(2m-n-2\)\\
&=-8\(2m - n - 2\),\allowdisplaybreaks\\
c_2&:=-4\(4m-n-6\)\(4m-2n-4\)+4\(4m-n-6\)\(6m-3n-2\)\\
&\quad-4\(4m-n-4\)\(2m-n\)\\
&=16\(m-3\),\allowdisplaybreaks\\
c_3&:=-4\(4m-n-6\)\(2m-n-2\)+2\(m-3\)\(4m-n-6\)\(2m-n-2\)\\
&\quad+4\(2m-n\)\(2m-n-2\)-2\(m-3\)\(2m-n-2\)\(4m-n-6\)\\
&=- 8\(m-3\)\(2m - n - 2\),\allowdisplaybreaks\\
c_4&:=-\(4m-n-6\)\(m-3\)\(4m-2n-4\)+4\(4m-n-6\)\(3m-n-3\)\\
&\quad-4\(4m-n-6\)\(2m-n\)+2\(m-3\)\(4m-n-6\)\(2m-n-2\)\\
&=4\(m-3\)\(4m-n-6\),\allowdisplaybreaks\\
c_5&:=\(m-3\)\(2m-n-2\)^2,\allowdisplaybreaks\\
c_6&:=-2\(m-3\)\(4m-n-6\)\(2m-n-2\),\allowdisplaybreaks\\
c_7&:=\(m-3\)\(4m-n-6\)\(3m-n-3\).   
\end{align*}
To compute $ \(-\Delta\)^3 \(u^t\) $, we write
\begin{align}\label{laplacian cubed}
 \Delta^3\(u^t\)&=\Delta\(u^{t-4}\cdot u^{4-t} \Delta^2\(u^t\)\)\nonumber\\
&= u^{4-t}\Delta\(u^{t-4}\)\Delta^2\(u^t\)+2\<\nabla\(u^{t-4}\),\nabla\(u^{4-t} \Delta^2\(u^t\)\) \>\nonumber\\
 &\quad+ u^{t-4} \Delta \(u^{4-t} \Delta^2\(u^t\)\).
\end{align}
Similarly to \eqref{laplacian for u^{t-4}}, we now write
$$ t - 4 = \frac{ 8m-3n-6}{n-2m}\quad\text{and}\quad t-5 = \frac{10m-4n-6}{n-2m},$$
and calculate
\begin{align}\label{laplacian for u^{t-4}}
&u\(x\)^{6-t}\Delta\(u^{t-4}\)\(x\)\nonumber\\
&\quad=\(8m-3n-6\)\int_{\mathbb{R}^{2n}} \left|x - y_5\right|^{2m-n}\left|x - y_6\right|^{2m-n}{\(\(8m-4n-4\)Z_5\(x,y\)\right.}\nonumber\\
&\qquad{\left.-\(5m-2n-3\)Z_{5,6}\(x,y\)\)d\mu_g^{\(2m\)}\(y\)}.
\end{align}
Applying (\ref{general laplacian}), (\ref{general inner product}), (\ref{formula for laplacian squared}), and (\ref{laplacian for u^{t-4}}) gives
\begin{align*}
&\frac{1}{n-6}\,u\(x\)^{10-2t}\Delta\(u^{t-4}\)\(x\)\Delta^2\(u^t\)\(x\)\\
&\quad=\(8m-3n-6\)\int_{\mathbb{R}^{6n}} \prod_{i=1}^6 \left|x - y_i\right|^{2m - n}\big(c_1Z_1^2+ c_2 Z_1 Z_2+c_3 Z_1 Z_{1,2}\\
&\qquad+ c_4 Z_1 Z_{2,3}+c_5 Z_{1,2}^2+c_6Z_{1,2} Z_{1,3}+ c_7 Z_{1,2} Z_{3,4}\big)\(\(8m-4n-4\)Z_5\right.\\
&\qquad\left.-\(5m-2n-3\)Z_{5,6}\){\(x,y\)d\mu_g^{\(2m\)}\(y\)},\allowdisplaybreaks\\
&\frac{2}{n-6}\,u\(x\)^{6-t}\<\nabla\(u^{t-4}\)\(x\),\nabla\(u^{4-t} \Delta^2\(u^t\)\)\(x\)\>\\
&\quad=\(3n-8m+6 \)\int_{\mathbb{R}^{6n}} \prod_{i=1}^6 \left|x - y_i\right|^{2m - n}\big(c_1Z_1^2\big(\(2m-n-4\)\(Z_1+Z_5\right.\\
&\qquad\left.-Z_{1,5}\)+\(2m-n\)\(Z_2+Z_3+Z_4+3Z_5-Z_{2,5}-Z_{3,5}-Z_{4,5}\)\big)\allowdisplaybreaks\\
&\qquad+ c_2 Z_1 Z_2\big(\(2m-n-2\)\(Z_1+Z_2+2Z_5-Z_{1,5}-Z_{2,5}\)+\(2m-n\)\\
&\qquad\times\(Z_3+Z_4+2Z_5-Z_{3,5}-Z_{4,5}\)\big)+c_3 Z_1 Z_{1,2}\big(\(2m-n-4\)\(Z_1+Z_5\right.\allowdisplaybreaks\\
&\qquad\left.-Z_{1,5}\)+\(2m-n-2\)\(Z_2+Z_5-Z_{2,5}\)+\(2m-n\)\(Z_3+Z_4+2Z_5\right.\allowdisplaybreaks\\
&\qquad\left.-Z_{3,5}-Z_{4,5}\)\big) + c_4 Z_1 Z_{2,3}\big(\(2m-n-2\)\(Z_1+Z_2+Z_3+3Z_5-Z_{1,5}\right.\allowdisplaybreaks\\
&\qquad\left.-Z_{2,5}-Z_{3,5}\)+\(2m-n\)\(Z_4+Z_5-Z_{4,5}\)\big)+c_5 Z_{1,2}^2\big(\(2m-n-4\)\\
&\qquad\times\(Z_1+Z_2+2Z_5-Z_{1,5}-Z_{2,5}\)+\(2m-n\)\(Z_3+Z_4+2Z_5-Z_{3,5}\right.\allowdisplaybreaks\\
&\qquad\left.-Z_{4,5}\)\big)+c_6Z_{1,2} Z_{1,3}\big(\(2m-n-4\)\(Z_1+Z_5-Z_{1,5}\)+\(2m-n-2\)\\
&\qquad\times\(Z_2+Z_3+2Z_5-Z_{2,5}-Z_{3,5}\)+\(2m-n\)\(Z_4+Z_5-Z_{4,5}\)\big)\allowdisplaybreaks\\
&\qquad+ \(2m-n-2\)c_7  Z_{1,2} Z_{3,4}\(Z_1+Z_2+Z_3+Z_4+4Z_5-Z_{1,5}-Z_{2,5}\right.\\
&\qquad\left.-Z_{3,5}-Z_{4,5}\)\big){\(x,y\)d\mu_g^{\(2m\)}\(y\)},
\end{align*}
and
\begin{align*}
&\frac{1}{n-6}\,u\(x\)^2 \Delta \(u^{4-t} \Delta^2\(u^t\)\)\(x\)\\
&\quad=\int_{\mathbb{R}^{6n}} \prod_{i=1}^6 \left|x - y_i\right|^{2m - n}\big(c_1Z_1^2\big(\(8m-3n-6\)\(\(2m-n-4\)Z_1\right.\allowdisplaybreaks\\
&\qquad\left.+\(2m-n\)\(Z_2+Z_3+Z_4\)\)-\(2m-n\)\(2m-n-4\)\(Z_{1,2}+Z_{1,3}\right.\allowdisplaybreaks\\
&\qquad\left.+Z_{1,4}\)-\(2m-n\)^2\(Z_{2,3}+Z_{2,4}+Z_{3,4}\)\big)+ c_2 Z_1 Z_2\big(\(8m-3n-6\)\\
&\qquad\times\(\(2m-n-2\)\(Z_1+Z_2\)+\(2m-n\)\(Z_3+Z_4\)\)-\(2m-n-2\)^2Z_{1,2}\allowdisplaybreaks\\
&\qquad-\(2m-n\)\(2m-n-2\)\(Z_{1,3}+Z_{1,4}+Z_{2,3}+Z_{2,4}\)-\(2m-n\)^2Z_{3,4}\big)\allowdisplaybreaks\\
&\qquad+c_3 Z_1 Z_{1,2}\big(\(8m-3n-8\)\(\(2m-n-4\)Z_1+\(2m-n-2\)Z_2\right.\allowdisplaybreaks\\
&\qquad\left.+\(2m-n\)\(Z_3+Z_4\)\)-\(2m-n-2\)\(2m-n-4\)Z_{1,2}-\(2m-n\)\\
&\qquad\times\(2m-n-4\)\(Z_{1,3}+Z_{1,4}\)-\(2m-n\)\(2m-n-2\)\(Z_{2,3}+Z_{2,4}\)\allowdisplaybreaks\\
&\qquad-\(2m-n\)^2Z_{3,4}\big)+ c_4 Z_1 Z_{2,3}\big(\(8m-3n-8\)\(\(2m-n-2\)\(Z_1+Z_2\right.\right.\allowdisplaybreaks\\
&\qquad\left.\left.+Z_3\)+\(2m-n\)Z_4\)-\(2m-n-2\)^2\(Z_{1,2}+Z_{1,3}+Z_{2,3}\)-\(2m-n\)\\
&\qquad\times\(2m-n-2\)\(Z_{1,4}+Z_{2,4}+Z_{3,4}\)\big)+c_5 Z_{1,2}^2\big(\(8m-3n-10\)\\
&\qquad\times\(\(2m-n-4\)\(Z_1+Z_2\)+\(2m-n\)\(Z_3+Z_4\)\)-\(2m-n-4\)^2Z_{1,2}\allowdisplaybreaks\\
&\qquad-\(2m-n\)\(2m-n-4\)\(Z_{1,3}+Z_{1,4}+Z_{2,3}+Z_{2,4}\)-\(2m-n\)^2Z_{3,4}\big)\allowdisplaybreaks\\
&\qquad+c_6Z_{1,2} Z_{1,3}\big(\(8m-3n-10\)\(\(2m-n-4\)Z_1+\(2m-n-2\)\right.\\
&\qquad\left.\times\(Z_2+Z_3\)+\(2m-n\)Z_4\)-\(2m-n-2\)\(2m-n-4\)\(Z_{1,2}+Z_{1,3}\)\\
&\qquad-\(2m-n\)\(2m-n-4\)Z_{1,4}-\(2m-n-2\)^2Z_{2,3}-\(2m-n\)\\
&\qquad\times\(2m-n-2\)\(Z_{2,4}+Z_{3,4}\)\big)+ \(2m-n-2\)c_7 Z_{1,2} Z_{3,4}\\
&\qquad\times\big(\(8m-3n-10\)\(Z_1+Z_2+Z_3+Z_4\)-\(2m-n-2\)\(Z_{1,2}+Z_{1,3}\right.\allowdisplaybreaks\\
&\qquad\left.+Z_{1,4}+Z_{2,3}+Z_{2,4}+Z_{3,4}\)\big)\big){\(x,y\)d\mu_g^{\(2m\)}\(y\)}.
\end{align*}
Interchanging the order of integration and plugging these formulas into \eqref{laplacian cubed} then gives
\begin{align*}
Q_g^{\(6\)}\(x\)&=\frac{2}{n-6}\,u\(x\)^{6-t} \(-\Delta\)^3 \(u^t\)\(x\)\allowdisplaybreaks\\
&= 2\int_{\mathbb{R}^{6n}} \prod_{i=1}^6 \left|x - y_i\right|^{2m - n}Z_{1,2}{\(x,y\)}\big(c^\prime_1 Z_{1,2}^2+c^\prime_2Z_{1,2} Z_{1,3}\allowdisplaybreaks\\ 
&\quad+ c^\prime_3 Z_{1,2}Z_{3,4}+c^\prime_4Z_{1,3}Z_{1,4}+c^\prime_5Z_{1,3}Z_{2,3}+c^\prime_6Z_{1,3} Z_{3,4}\\
&\quad+c^\prime_7Z_{1,3} Z_{5,6}+c^\prime_8Z_{3,4} Z_{5,6}\big){\(x,y\)d\mu_g^{\(2m\)}\(y\)},
\end{align*}
where
\begin{align*}
c^\prime_1&:=\(2m-n-4\)^2c_5\\
&=\(m-3\)\(2m-n-2\)^2\(2m-n-4\)^2,\allowdisplaybreaks\\
c^\prime_2&:=2\(3n-8m+6\)\(2m-n-4\)c_5+4\(2m-n\)\(2m-n-4\)c_5\\
&\quad+2\(2m-n-2\)\(2m-n-4\)c_6\\
&=-6\(m-3\)\(4m - n - 6\)\(2m-n-2\)^2\(2m-n-4\),\allowdisplaybreaks\\
c^\prime_3&:=\(8m-3n-6\)\(5m-2n-3\)c_5+2\(3n-8m+6\)\(2m-n\)c_5\\
&\quad+\(2m-n\)^2c_5+2\(2m-n-2\)^2c_7\\
&=3\(m-3\)\(4m - n - 6\)\(3m - n - 3\)\(2m-n-2\)^2, \allowdisplaybreaks\\
c^\prime_4&:=\(3n-8m+6\)\(2m-n-4\)c_6+\(2m-n\)\(2m-n-4\)c_6\\
&=4\(m-3\)\(4m - n  - 6\)\(3m - n - 3\)\(2m - n - 2\)\(2m - n - 4\),\allowdisplaybreaks\\
c^\prime_5&:=\(2m-n-2\)^2c_6\\
&=-2\(m-3\)\(4m - n - 6\)\(2m - n - 2\)^3 , \allowdisplaybreaks\\
c^\prime_6&:=2\(3n-8m+6\)\(2m-n-2\)c_6+2\(2m-n\)\(2m-n-2\)c_6\\
&\quad+4\(2m-n-2\)^2c_7\\
&=12\(m-3\)\(4m - n - 6\)\(3m - n - 3\)\(2m-n-2\)^2, \allowdisplaybreaks\\
c^\prime_7&:=\(8m-3n-6\)\(5m-2n-3\)c_6+\(3n-8m+6\)\(2m-n\)c_6\\
&\quad+4\(3n-8m+6\)\(2m-n-2\)c_7\\
&=-6\(m-3\)\(4m - n - 6\)\(3m - n - 3\)\(8m-3n-6\)\(2m - n - 2\),\allowdisplaybreaks\\
c^\prime_8&:=\(8m-3n-6\)\(5m-2n-3\)c_7\\
&=\(m-3\)\(4m - n - 6\)\(3m - n - 3\)\(8m-3n-6\)\(5m-2n-3\).
\end{align*}
Note that no term containing $Z_1,\dotsc,Z_6$ appears in this formula. Indeed, computations similar to those above show that the coefficients of all these terms are zero.

\smallskip
Now we assume $ n = 2m $. In this case, in addition to (\ref{general laplacian}), (\ref{laplacian}), and (\ref{general inner product}), we note the formula
\begin{align}\label{change of dimension}
\int_{\R^{\(k+l\)n}} {f\(y\)d\mu_g^{\(n\)}\(y\)}= \lambda^k  \int_{\R^{ln}} {f\(y\)d\mu_g^{\(n\)}\(y\)}
\end{align}
for all $k,l\in\N$, $y=\(y_1,\dotsc,y_l\)\in\R^{ln}$ and $f\in {L^1\big(\R^{ln},\mu_g^{\(n\)}\big)}$. We compute
\begin{align}\label{laplacian for exponential}
&-{e^{-\frac{n-6}{2} u\(x\)}}\Delta \( e^{\frac{n-6}{2} u}\)\(x\)\nonumber\\
&\quad=  -\frac{n-6}{2} \( \Delta u\(x\) +  \frac{n-6}{2}  \left|\nabla u\(x\)\right|^2 \)\allowdisplaybreaks\nonumber\\
&\quad=\frac{n-6}{2}\(\(n-2\)\int_{\mathbb{R}^n}{Z_1\,d\mu_g^{\(n\)}}-\frac{n-6}{4}\int_{\mathbb{R}^{2n}}{\(Z_1+Z_2-Z_{1,2}\)d\mu_g^{\(n\)}}\)\nonumber\\
&\quad=a_1\int_{\mathbb{R}^n}{Z_1\,d\mu_g^{\(2m\)}}+  a_2 \int_{\mathbb{R}^{2n}}{Z_{1,2}\,d\mu_g^{\(n\)}},
\end{align}
where 
$$a_1:=\frac{n-6}{2}\(n-2 - \frac{n-6}{2}\lambda \)\quad\text{and}\quad a_2:=\frac{\(n-6\)^2}{8}.$$
We then compute
\begin{align}\label{laplacian squared for exponential}
&e^{-\frac{n-6}{2} u\(x\)}\Delta^2\(e^{\frac{n-6}{2} u}\)\(x\)\nonumber\\
&\quad=e^{-\(n-6\) u\(x\)}\(\Delta\(e^{\frac{n-6}{2} u}\)\(x\)\)^2+ \Delta\(e^{-\frac{n-6}{2} u}\Delta\(e^{\frac{n-6}{2} u}\)\(x\)\)\nonumber\\
&\qquad  + \(n-6\)\< \nabla u\(x\), \nabla\(e^{-\frac{n-6}{2} u}\Delta\(e^{\frac{n-6}{2} u}\)\)\(x\)\>\allowdisplaybreaks\nonumber\\
&\quad=a_1^2\int_{\R^{2n}}Z_1Z_2{\,d\mu_g^{\(n\)}}+2a_1\int_{\R^{3n}}Z_1Z_{2,3}{\,d\mu_g^{\(n\)}}+a_2^2\int_{\R^{4n}}Z_{1,2}Z_{3,4}{\,d\mu_g^{\(n\)}}\nonumber\\
&\qquad+2\(n-4\)a_1\int_{\R^n}Z_1^2{\,d\mu_g^{\(n\)}}+a_2\int_{\R^{2n}}Z_{1,2}\(2\(n-6\)\(Z_1+Z_2\)\right.\nonumber\allowdisplaybreaks\\
&\qquad\left.+4Z_{1,2}\){d\mu_g^{\(n\)}}-\(n-6\)\(a_1\int_{\R^{2n}}Z_1\(Z_1+Z_2-Z_{1,2}\){d\mu_g^{\(n\)}}\right.\nonumber\\
&\qquad\left.+a_2\int_{\R^{3n}}Z_{1,2}\(Z_1+Z_2+2Z_3-Z_{1,3}-Z_{2,3}\){d\mu_g^{\(n\)}}\)\allowdisplaybreaks\nonumber\\
&\quad= b_1 \int_{\R^n} Z_1^2{\,d\mu_g^{\(n\)}} + b_2 \int_{\R^{2n}} Z_1 Z_2{\,d\mu_g^{\(n\)}} + b_3 \int_{\R^{2n}} Z_{1}Z_{1,2}{\,d\mu_g^{\(n\)}}\nonumber\allowdisplaybreaks\\ 
&\qquad + b_4 \int_{\R^{2n}} Z_{1,2}^2{\,d\mu_g^{\(n\)}}+ b_5 \int_{\R^{3n}} Z_1 Z_{2,3}{\,d\mu_g^{\(n\)}}+ b_6 \int_{\R^{3n}} Z_{1,2} Z_{1,3}{\,d\mu_g^{\(n\)}}\nonumber\\
&\qquad +  b_7 \int_{\R^{4n}} Z_{1,2} Z_{3,4}{\,d\mu_g^{\(n\)}},
\end{align}
where
\begin{align*}
b_1 &:=\(2\(n-4\) - \(n-6\)\lambda\)a_1 \\
& = \(n-6\)\(n - 2 - \frac{n-6}{2}\lambda\)\(n - 4 - \frac{n-6}{2}\lambda\), \allowdisplaybreaks \\
b_2 &:=a_1^2-\(n-6\)a_1\\
&= \frac{\(n-6\)^2}{4}\(n - 2 - \frac{ n-6}{2}\lambda\)\(n - 4 - \frac{n-6}{2}\lambda\), \allowdisplaybreaks\\
b_3 &:=  \(n-6\)a_1+\(n-6\)\(4-2\lambda\)a_2 \\
&= \(n-6\)^2\(n - 4 -  \frac{n-6}{2}\lambda\),\allowdisplaybreaks \\
b_4 &:= 4a_2\\
&=\frac{\(n-6\)^2}{2},\allowdisplaybreaks \\
b_5 &:=  2a_1a_2-2\(n-6\)a_2\\
&= \frac{\(n-6\)^3}{8}\(n - 4 -  \frac{n-6}{2}\lambda\), \allowdisplaybreaks \\
b_6 &:=2\(n-6\)a_2\\
&= \frac{\(n-6\)^3}{4},\allowdisplaybreaks \\
b_7 &:=a_2^2\\
&= \frac{\(n-6\)^4}{64}.
\end{align*}
Applying (\ref{general laplacian}), (\ref{general inner product}), and \eqref{laplacian squared for exponential} then gives
\begin{align*}
&2  \< \nabla u\(x\), \nabla\(e^{-\frac{n-6}{2} u} \Delta^2 \(e^{\frac{n-6}{2} u}\)\)\(x\)\>\\
&\quad=  4 b_1\int_{\R^{2n}}  Z_1^2 \(Z_1 + Z_2 - Z_{1,2}\){d\mu_g^{\(n\)}}+   \int_{\R^{3n}} \big( 2 b_2 Z_1 Z_2 \( Z_1 +Z_2 +  2Z_3\right.\allowdisplaybreaks\\
&\qquad\left. - Z_{1,3}  - Z_{2,3}\)+ b_3 Z_1 Z_{1,2} \( 4\(Z_1 + Z_3 - Z_{1,3}\) + 2\(Z_2 + Z_3 - Z_{2,3}\)\)\allowdisplaybreaks\\
 &\qquad+ 4b_4 Z_{1,2}^2\(Z_1 + Z_2 + 2Z_3 - Z_{1,3} - Z_{2,3}\)\big){\,d\mu_g^{\(n\)}}+ \int_{\R^{4n}} \big( 2 b_5Z_1 Z_{2,3}\\
&\qquad\times\(Z_1 + Z_2 + Z_3 + 3Z_4 - Z_{1,4} - Z_{2,4} - Z_{3,4}\)+ b_6 Z_{1,2} Z_{1,3} \( 4\(Z_1 + Z_4\right.\right.\allowdisplaybreaks\\
&\qquad\left.\left.- Z_{1,4}\) + 2\(Z_2 + Z_3 + 2Z_4 - Z_{2,4} - Z_{3,4}\) \) \big){\,d\mu_g^{\(n\)}}+ 2 b_7 \int_{\R^{5n}} Z_{1,2}Z_{3,4}\\
&\qquad\times\(Z_1 + Z_2 + Z_3 + Z_4 + 4Z_5 - Z_{1,5} - Z_{2,5} - Z_{3,5} - Z_{4,5}\){d\mu_g^{\(n\)}}
\end{align*}
and
\begin{align*}
&\Delta\(e^{-\frac{n-6}{2} u} \Delta^2 \(e^{\frac{n-6}{2} u}\)\)\(x\)\\
 &\quad= -4\(n-6\) b_1 \int_{\R^n} Z_1^3{\,d\mu_g^{\(n\)}}- \int_{\R^{2n}} \big( b_2 Z_1 Z_2 \( 2\(n-6\)\(Z_1 + Z_2\) + 4Z_{1,2} \)   \allowdisplaybreaks\\
&\qquad + b_3 Z_1 Z_{1,2} \( \(n-8\)\(4Z_1 + 2Z_2\) + 8Z_{1,2} \)+ b_4 Z_{1,2}^2 \( 4\(n-10\)\(Z_1 + Z_2\)\right.\allowdisplaybreaks\\
&\qquad\left.+ 16Z_{1,2}\)\big){\,d\mu_g^{\(n\)}}- \int_{\R^{3n}} \big( b_5 Z_1 Z_{2,3} \( 2\(n-8\)\(Z_1 + Z_2 + Z_3\)\right.   \allowdisplaybreaks\\
&\qquad\left.+ 4\(Z_{1,2} + Z_{1,3} + Z_{2,3}\) \)+ b_6 Z_{1,2} Z_{1,3} \( \(n-10\)\(4Z_1 + 2Z_2 + 2Z_3\)\right.   \allowdisplaybreaks\\
&\qquad\left. + 8 Z_{1,2} + 8Z_{1,3} + 4Z_{2,3} \) \big){\,d\mu_g^{\(n\)}}- b_7\int_{\R^{4n}} \big( Z_{1,2} Z_{3,4} \( 2\(n-10\)\(Z_1 + Z_2\right.\right.\\
&\qquad\left.\left. + Z_3 + Z_4\)+4\(Z_{1,2} + Z_{1,3} + Z_{1,4} + Z_{2,3} + Z_{2,4} + Z_{3,4}\) \)\big){\,d\mu_g^{\(n\)}}.
\end{align*}
Finally, applying these formulas together with \eqref{change of dimension} and interchanging the order of integration give
\begin{align*}
&\(-\Delta\)^3 \(e^{\frac{n-6}{2}u}\)\(x\)\\
&\quad= -e^{-\frac{n-6}{2} u\(x\)}\Delta\(e^{\frac{n-6}{2} u}\)\(x\)\Delta^2 \(e^{\frac{n-6}{2} u}\)\(x\)\\
&\qquad - \(n-6\)e^{\frac{n-6}{2}u\(x\)}\< \nabla u\(x\), \nabla\(e^{-\frac{n-6}{2} u} \Delta^2 \(e^{\frac{n-6}{2} u}\)\)\(x\) \>\\
&\qquad-e^{\frac{n-6}{2}u\(x\)}\Delta\(e^{-\frac{n-6}{2} u} \Delta^2 \(e^{\frac{n-6}{2} u}\)\)\(x\)  \\
&\quad= e^{\frac{n-6}{2}u\(x\)} \int_{\R^{6n}}\big(b'_1 Z_1^3 + b'_2 Z_1^2 Z_{1,2} + b'_3 Z_1^2 Z_2 + b'_4 Z_1^2 Z_{2,3}+ b'_5  Z_1 Z_{1,2}^2\\
&\qquad + b'_6 Z_1 Z_{1,2} Z_{1,3} + b'_7 Z_1 Z_2 Z_{1,2} + b'_8 Z_1 Z_{1,2} Z_{2,3}+ b'_9 Z_1 Z_{1,2} Z_{3,4}  \\
&\qquad + b'_{10} Z_1 Z_2 Z_{1,3} + b'_{11} Z_1 Z_2 Z_3 + b'_{12} Z_1 Z_2 Z_{3,4}+b'_{13}  Z_1 Z_{2,3}^2  \\
&\qquad + b'_{14} Z_1 Z_{2,3} Z_{2,4} + b'_{15} Z_1 Z_{2,3} Z_{4,5}  + b'_{16} Z_{1,2}^3+ b'_{17} Z_{1,2}^2 Z_{1,3}  \\
&\qquad + b'_{18} Z_{1,2}^2 Z_{3,4} + b'_{19} Z_{1,2} Z_{1,3} Z_{1,4} + b'_{20} Z_{1,2} Z_{1,3} Z_{2,3}  + b'_{21} Z_{1,2} Z_{1,3} Z_{3,4}\\
&\qquad+ b'_{22} Z_{1,2} Z_{1,3} Z_{5,6} + b'_{23} Z_{1,2} Z_{3,4} Z_{5,6}\big){\,d\mu_g^{\(n\)}}, 
\end{align*}
where
\begin{align*}
b'_1 &:= 2\(n-6\)\(2-\lambda\)\lambda^{-5} b_1, \allowdisplaybreaks\\
b'_2 &:= \lambda^{-4}\(2\(n-6\) b_1 + \(4\(n-8\) - 2\(n-6\)\lambda\)b_3\),\allowdisplaybreaks\\
b'_3 &:= \lambda^{-4}\(\(a_1 - 2\(n-6\)\) b_1 + \(n-6\)\(4 - 2\lambda\)b_2\),\allowdisplaybreaks\\
b'_4 &:= \lambda^{-3}\(a_2 b_1 + \(2\(n-8\) -  \(n-6\)\lambda\) b_5\),\allowdisplaybreaks\\
b'_5 &:=\lambda^{-4}\( 8 b_3 + \(8\(n-10\) - 4\(n-6\)\lambda\) b_4\),\allowdisplaybreaks\\
b'_6 &:= \lambda^{-3}\(2\(n-6\) b_3 +  \(4\(n-10\) - 2\(n-6\)\lambda\) b_6\),\allowdisplaybreaks\\
b'_7 &:= \lambda^{-4}\(4b_2 + \(2\(n-8\) - \(n-6\)\lambda\) b_3\),\allowdisplaybreaks\\
b'_8 &:= \lambda^{-3}\(\(n-6\) b_3 + 8 b_5 + \(4\(n-10\) - 2\(n-6\)\lambda\) b_6\),  \allowdisplaybreaks\\
b'_9 &:= \lambda^{-2}\(a_2b_3 + \(n-6\)b_5 + \(8\(n-10\)-4\(n-6\)\lambda\)b_7\),  \allowdisplaybreaks\\
b'_{10} &:= \lambda^{-3}\(2\(n-6\)b_2 + \(a_1-3\(n-6\)\)b_3 + \(4\(n-8\)-2\(n-6\)\lambda\)b_5\), \allowdisplaybreaks   \\
b'_{11} &:= \(a_1-2\(n-6\)\)\lambda^{-3}b_2,\allowdisplaybreaks\\
b'_{12} &:= \lambda^{-2}\(a_2b_2 + \(a_1-3\(n-6\)\)b_5\), \allowdisplaybreaks  \\
b'_{13} &:=\lambda^{-3}\( \(a_1-4\(n-6\)\)b_4 + 4b_5\), \allowdisplaybreaks  \\
b'_{14} &:=\lambda^{-2}\( 2\(n-6\)b_5 + \(a_1-4\(n-6\)\)b_6\),  \allowdisplaybreaks  \\
b'_{15} &:= \lambda^{-1}\(a_2b_5 + \(a_1-4\(n-6\)\)b_7\), \allowdisplaybreaks  \\
b'_{16} &:= 16\lambda^{-4}b_4, \allowdisplaybreaks  \\
b'_{17} &:= \lambda^{-3}\(4\(n-6\)b_4 + 16b_6\), \allowdisplaybreaks  \\
b'_{18} &:= \lambda^{-2}\(a_2b_4 + 8b_7\),  \allowdisplaybreaks  \\
b'_{19} &:= 2\(n-6\)\lambda^{-2}b_6, \allowdisplaybreaks  \\
b'_{20} &:= 4\lambda^{-3}b_6, \allowdisplaybreaks  \\
b'_{21} &:= \lambda^{-2}\(2\(n-6\)b_6 + 16b_7\), \allowdisplaybreaks  \\
b'_{22} &:= \lambda^{-1}\(a_2b_6 + 4\(n-6\)b_7\), \allowdisplaybreaks  \\
b'_{23} &:= a_2b_7.
\end{align*}
Plugging the expressions of $b_1,\dotsc,b_7$ into those of $b'_1,\dotsc,b'_{23}$ and factorizing then give \eqref{formula laplacian cubed n equals 2m}. This ends the proof of Proposition~\ref{Pr}.
\endproof

\section{Proof of Theorem~\ref{Th1}}\label{Sec3}

In this section, we prove Theorem~\ref{Th1} by using the integral formula for $Q_g^{\(6\)}$ obtained in the previous section:

\proof[Proof of Theorem~\ref{Th1}]
Let $m\ge4$, $n\in\left\{2m,2m + 1,\dotsc,4m-6\right\}$, $g$ be a conformally Euclidean metric on $\R^n$ satisfying (A), and $P_{n,m}$ be as in Proposition~\ref{Pr}. In the case where $n=2m$,
we observe that the last eight coefficients of $P_{2m,m}$ are positive and the first fifteen coefficients are nonnegative since $0<\lambda\le2$. Since $Q_g^{\(2m\)}\not\equiv0$ in $\R^n$ as assumed in (A), we then obtain $Q_g^{\(6\)}>0$ as a direct consequence of Proposition~\ref{Pr}. We now consider the case $n>2m$. In this case, we have $Q_g^{\(2m\)}>0$ in $\R^n\setminus B\(0,R\)$ as a consequence of (A). Analyzing the sign changes of the terms of $ P_{n, m}$ reveals three nontrivial cases to consider: 
$ \frac{5m - 3}{2} < n < \frac{8m-6}{3} $, $ \frac{8m-6}{3} < n < 3m - 3 $, and $ 3m - 3 < n < 4m - 6 $. We label these cases a), b), and c) respectively. Let 
$$ Y_{i,j}\(x,y\) := \frac{ x - y_i}{\left|x - y_i\right|^2} - \frac{ x - y_j}{\left|x - y_j\right|^2}$$
for all $i,j\in\left\{1,\dotsc,6\right\}$ and $x,y_1,\dotsc,y_6\in\R^n$. It is easy to see $ \left|Y_{i,j}\right|^2 = Z_{i,j} $. 

\smallskip
For case a), by expanding and interchanging the order of integration, we easily see
\begin{align}\label{identity 3456}
0&\le\int_{\mathbb{R}^{6n}} \prod_{i=1}^6 \left|x - y_i\right|^{2m - n} \left|Y_{1,2}\right|^2 \(\left|Y_{3,4}\right|^2 - \left|Y_{5,6}\right|^2\)^2{d\mu_g^{\(2m\)}\(y\)} \nonumber\allowdisplaybreaks\\
 &= 2 \int_{\mathbb{R}^{6n}} \prod_{i=1}^6 \left|x - y_i\right|^{2m - n} \left|Y_{1,2}\right|^2  \left|Y_{3,4}\right|^2\(\left|Y_{1,2}\right|^2 - \left|Y_{5,6}\right|^2\){d\mu_g^{\(2m\)}\(y\)}.
\end{align}
Therefore, analyzing the coefficients of $P_{n,m}$, to obtain $Q_g^{(6)}>0$ in $\R^n$, it suffices to show
$$3\(n-2m + 2\)^2\ge \(2n-5m+3\)\(8m - 3n - 6\) .$$
This is easy to see by noting the discriminant of
\begin{align*}
&3\(2m-n-2\)^2-\(2n-5m+3\)\(8m-3n-6\)\\
&\quad=9n^2 + \(33 - 43m\)n + 52m^2 - 78m + 30 
\end{align*}
as a quadratic in $ n $ is negative for $ m \geq 1 $.

\smallskip
For case b), we use the inequality
\begin{align*}
0&\le\int_{\mathbb{R}^{6n}} \prod_{i=1}^6 \left|x - y_i\right|^{2m - n} \left|Y_{1,2}\right|^2 \(\left|Y_{1,3}\right|^2 - \left|Y_{5,6}\right|^2\)^2{d\mu_g^{\(2m\)}\(y\)}   \allowdisplaybreaks\\
&=   \int_{\mathbb{R}^{6n}} \prod_{i=1}^6 \left|x - y_i\right|^{2m - n} \left|Y_{1,2}\right|^2 \Big( \left|Y_{1,2}\right|^2\left|Y_{1,3}\right|^2 + \left|Y_{1,2}\right|^2\left|Y_{3,4}\right|^2\\
&\qquad- 2\left|Y_{1,3}\right|^2\left|Y_{5,6}\right|^2\Big){d\mu_g^{\(2m\)}\(y\)} 
 \end{align*}
obtained by Fubini's theorem similarly to (\ref{identity 3456}). To obtain $Q_g^{(6)}>0$ in $\R^n$, it therefore suffices to show the inequalities
$$ 2\(n - 2m + 4\)\(n-2m + 2\) \geq \(3n-8m+6\)\(3m - n - 3\) $$
and
$$ n-2m+2 \geq 3n-8m+6 .$$
The first is obtained by arguing as in case a) and the second is immediate from $ 3m - 3 > n $.

\smallskip
For case c), we have four negative terms to consider. We therefore rearrange the terms of $P_{n,m}$ as
\begin{align*}
&P_{n,m}\(Z\)\\
&\quad=Z_{1,2}\big(\(n - 2m + 2\)^2\(n - 2m + 4\)^2 Z_{1,2}^2- \(4m - n - 6\)\(n  - 3m + 3\)\\
&\qquad\times\big(9\(n-2m+2\)^2+ \(3n-8m+6\)\(2n-5m+3\) - 3\(n-2m + 2\)\\
&\qquad\times\(3n-8m+6\)\big)Z_{1,2}Z_{3,4}+ \(4m - n - 6\)\(n-2m+2\)\(6\(m+1\)\right.\\
&\qquad\quad\times\left.\(n-2m+2\)Z_{1,2} Z_{1,3}-4\(n  - 3m + 3\)\(n - 2m + 4\)Z_{1,3}Z_{1,4}\)\allowdisplaybreaks\\
&\qquad+2\(4m - n - 6\)\(n - 2m + 2\)^3Z_{1,3}Z_{2,3}+ 6\(4m - n - 6\)\(n  - 3m + 3\)\\
&\qquad\times\(n-2m+2\)^2\(Z_{1,2}Z_{3,4}+Z_{1,2} Z_{1,3}  -2 Z_{1,3} Z_{3,4}\)+ 3\(4m - n - 6\) \\
&\qquad\times\(3m - n - 3\)\(8m-3n-6\)\(n - 2m + 2\)\(2Z_{1,3} Z_{5,6} - Z_{1,2}Z_{3,4}\)\allowdisplaybreaks\\
&\qquad+ \(4m - n - 6\)\(n-3m+3\)\(3n-8m+6\)\(2n-5m+3\)\\
&\qquad\times\(Z_{1,2}Z_{3,4} -Z_{3,4} Z_{5,6}\)\big).
\end{align*}
We apply the inequalities
\begin{align*}
0&\le\int_{\mathbb{R}^{6n}} \prod_{i=1}^6 \left|x - y_i\right|^{2m - n} \left|Y_{1,2}\right|^2 \(\left|Y_{1,3}\right|^2 - \left|Y_{3,4}\right|^2\)^2{d\mu_g^{\(2m\)}\(y\)}\allowdisplaybreaks\\
 &=  \int_{\mathbb{R}^{6n}} \prod_{i=1}^6 \left|x - y_i\right|^{2m - n} \left|Y_{1,2}\right|^2 \Big( \left|Y_{1,2}\right|^2\left|Y_{1,3}\right|^2 + \left|Y_{1,2}\right|^2\left|Y_{3,4}\right|^2\\
 &\quad- 2\left|Y_{1,3}\right|^2\left|Y_{3,4}\right|^2\Big){d\mu_g^{\(2m\)}\(y\)},\allowdisplaybreaks\\
0&\le\int_{\mathbb{R}^{6n}} \prod_{i=1}^6 \left|x - y_i\right|^{2m - n} \left|Y_{1,2}\right|^2 \(\left|Y_{1,3}\right|^2 - \left|Y_{1,4}\right|^2\)^2{d\mu_g^{\(2m\)}\(y\)} \allowdisplaybreaks\\
&= 2 \int_{\mathbb{R}^{6n}} \prod_{i=1}^6 \left|x - y_i\right|^{2m - n} \left|Y_{1,2}\right|^2  \left|Y_{1,3}\right|^2\(\left|Y_{1,2}\right|^2 - \left|Y_{1,4}\right|^2\){d\mu_g^{\(2m\)}\(y\)},\allowdisplaybreaks\\
0&\le \int_{\mathbb{R}^{6n}} \prod_{i=1}^6 \left|x - y_i\right|^{2m - n} \left|Y_{1,2}\right|^2 \(\left|Y_{1,2}\right|^2 - \left|Y_{3,4}\right|^2\)^2{d\mu_g^{\(2m\)}\(y\)} \allowdisplaybreaks\\
&=  \int_{\mathbb{R}^{6n}} \prod_{i=1}^6 \left|x - y_i\right|^{2m - n} \left|Y_{1,2}\right|^4 \(\left|Y_{1,2}\right|^2 - \left|Y_{3,4}\right|^2\){d\mu_g^{\(2m\)}\(y\)},
\end{align*}
and
\begin{align*}
0&\le\int_{\mathbb{R}^{6n}} \prod_{i=1}^6 \left|x - y_i\right|^{2m - n} \left|Y_{1,2}\right|^2 \(\left|Y_{3,4}\right|^2\left|Y_{5,6}\right|^2 - \< Y_{3,4}, Y_{5,6} \>^2 \){d\mu_g^{\(2m\)}\(y\)}\allowdisplaybreaks\\
&=\int_{\mathbb{R}^{6n}} \prod_{i=1}^6 \left|x - y_i\right|^{2m - n} \left|Y_{1,2}\right|^2\bigg(\left|Y_{3,4}\right|^2\left|Y_{5,6}\right|^2\\
&\quad-\frac{1}{4}\(\left|Y_{3,5}\right|^2+\left|Y_{4,6}\right|^2-\left|Y_{3,6}\right|^2-\left|Y_{4,5}\right|^2\)^2\bigg){d\mu_g^{\(2m\)}\(y\)}\allowdisplaybreaks\\
&=  \int_{\mathbb{R}^{6n}} \prod_{i=1}^6 \left|x - y_i\right|^{2m - n}{ \left|Y_{1,2}\right|^2\(2 \left|Y_{1,3}\right|^2  \left|Y_{5,6}\right|^2 - \left|Y_{1,2}\right|^2 \left|Y_{3,4}\right|^2\)}{d\mu_g^{\(2m\)}\(y\)}
 \end{align*}
as well as $ (\ref{identity 3456}) $. To obtain $Q_g^{(6)}>0$ in $\R^n$, it therefore suffices to prove the inequalities
\begin{align*} 
 &\(n - 2m + 2\)^2\(n - 2m + 4\)^2 \\
 &\quad\geq \(4m - n - 6\)\(n-3m+3\)\big(9\(n-2m+2\)^2\\
&\qquad+ \(3n-8m+6\)\(2n-5m+3\) - 3\(n-2m + 2\)\(3n-8m+6\)\big)\\
&\quad=\(4m - n - 6\)\(n-3m+3\)\(6n^2 -\(25m-21\)n + 28m^2 - 42m + 18\)
\end{align*}
and
$$6\(m+1\)\(n-2m + 2\) \geq 4\(n - 2m + 4\)\(n-3m+3\).$$
To obtain the first inequality, it is easy to see 
$$ 6n^2 - 25mn + 21n + 28m^2 - 42m + 18 \leq 7\(n - 2m + 4\)^2 $$
in the range $ 3m - 3 < n < 4m - 6 $ by writing
\begin{align*}
&7\(n - 2m + 4\)^2-\(6n^2 - 25mn + 21n + 28m^2 - 42m + 18\)\\
&\quad=n^2 - \(3m - 35\)n - 70m + 94\\
&\quad=\int_{3m-3}^n\(2t-3m+35\)dt+26m-2.
\end{align*}
We then conclude by noting 
\begin{align*}
&\(n - 2m + 2\)^2 - 7\(4m - n - 6\)\(n-3m+3\)\\
&\quad=8n^2-\(53m-67\)n+88m^2-218m+130\\
&\quad > 0,
\end{align*}
as its discriminant as a quadratic form in $n$ is negative for $ m \geq 4$. For the second inequality, since $ 3m - 3 < n < 4m - 6 $ and $m\ge4$, we have
\begin{align*}
&6\(m + 1\)\(n - 2m + 2\)-4\(n - 2m + 4\)\(n-3m+3\)\\
&\quad>6\(m + 1\)\(n - 2m + 2\)-4\(m-3\)\(n - 2m + 4\)\allowdisplaybreaks\\
&\quad=2\(m+9\)n-4m^2-40m+60\allowdisplaybreaks\\
&\quad>2\(m+9\)\(3m-3\)-4m^2-40m+60\allowdisplaybreaks\\
&\quad=2\(m+1\)\(m+3\)\\
&\quad>0.
\end{align*}
Therefore, in all cases a), b), and c), we obtain $Q_g^{\(6\)}>0$ in $\R^n$. This ends the proof of Theorem~\ref{Th1}.
\endproof

\section{Proof of Theorem~\ref{Th2}}\label{Sec4}

In this section, we prove Theorem~\ref{Th2} by again using the integral formula for $Q_g^{\(6\)}$ from Section~\ref{Sec2}:

\proof[Proof of Theorem~\ref{Th2}]
Let $n,m\in\N$ be such that $n > 2m\ge8$ and $P_{n,m}$ be as in \eqref{formula laplacian cubed}. For $\varepsilon>0$, let $ \rho_\varepsilon \in C^\infty_c\(\mathbb{R}^n\) $ be such that $ \rho_\varepsilon \geq 0 $, $ \int_{\R^n} \rho_\varepsilon\(x\) dx = 1 $ and $ \rho_\varepsilon\mathcal{L}\overset{\ast}{\rightharpoonup} \delta_0 $ weakly in the sense of measures as $ \varepsilon \to 0 $, where $\mathcal{L}$ is the Lebesgue measure in $\R^n$. For $r,\varepsilon>0$ and $x\in\R^n$, let 
$$ h_{r,\varepsilon}\(x\):= \varepsilon \big(1 + \left|x\right|^2\big)^{-\frac{n+2m}{2}} + r \rho_\varepsilon\(x -e_1\) + \rho_\varepsilon\(x + e_1\),$$
where $e_1:=\(1,0,\dotsc,0\)$. In particular, it is easy to see that
$$h_{r,\varepsilon}\mathcal{L}\overset{\ast}{\rightharpoonup}r \delta_{e_1} + \delta_{-e_1}$$
as $ \varepsilon \to 0 $ weakly in the sense of measures, where $\delta_{\pm e_1}$ is the Dirac measure at $\pm e_1$. Now let 
$$ u_{r,\varepsilon}\(x\):=\frac{\Gamma\(\frac{n}{2}-m+1\)}{2^{2m-1}\(m-1\)\,!\,\Gamma\(\frac{n}{2}\)\omega_{n-1}}\int_{\mathbb{R}^n} \left|x - y\right|^{2m - n} h_{r,\varepsilon}\(y\) dy .$$
Then $ u_{r,\varepsilon} $ is smooth and satisfies 
$$\(-\Delta\)^m u_{r,\varepsilon}= \frac{n-2m}{2}\,h_{r,\varepsilon}\quad\text{in }\R^n$$
and
$$ C_{r,\varepsilon}^{-1} \big(1 + \left|x\right|^2\big)^{\frac{2m - n}{2}} \leq  u_{r,\varepsilon}\(x\) \leq C_{r,\varepsilon} \big(1 + \left|x\right|^2\big)^{\frac{2m - n}{2}}$$
for some constant $C_{r,\varepsilon}>0$ depending on $n$, $m$, $r$ and $\varepsilon$. Let $g_{r,\varepsilon}$ be the metric on $\R^n$ defined by 
$$g_{r,\varepsilon} := u_{r,\varepsilon}^\frac{4}{n-2m}\left|dx\right|^2.$$
Then $Q_{g_{r,\varepsilon}}^{\(2m\)}=u_{r,\varepsilon}^{-\frac{n+2m}{n-2m}}h_{r,\varepsilon}$ satisfies
$$C_{r,\varepsilon}^{-1} \leq Q_{g_{r,\varepsilon}}^{\(2m\)}\leq C_{r,\varepsilon},$$
for some constant $C_{r,\varepsilon}>0$ depending on $n$, $m$, $r$ and $\varepsilon$. In particular, $ g_{r,\varepsilon} $ satisfies (A) for all $ \varepsilon > 0$. Thus Proposition~\ref{Pr} gives, 
\begin{align}\label{Th2Eq2}
Q_{g_{r,\varepsilon}}^{\(6\)}\(0\)&=2\(m-3\)\int_{\mathbb{R}^{6n}} \prod_{i=1}^6  \left|y_i\right|^{2m - n} P_{n,m}\(Z\(0,y\)\)d{\mu_{g_{r,\varepsilon}}^{\(2m\)}}\(y\)\nonumber\\
&\to{\(\frac{\Gamma\(\frac{n}{2}-m+1\)}{2^{2m-1}\(m-1\)\,!\,\Gamma\(\frac{n}{2}\)\omega_{n-1}}\)^6}\sum_{\alpha \in \{e_1, -e_1\}^6} r^{\theta_\alpha} P_{n,m}\(Z\(0,\alpha\)\)\nonumber\\
&=:{\(\frac{\Gamma\(\frac{n}{2}-m+1\)}{2^{2m-1}\(m-1\)\,!\,\Gamma\(\frac{n}{2}\)\omega_{n-1}}\)^6}K_{n,m}\(r\)
\end{align}
as $\varepsilon\to0$, where 
$$\theta_\alpha := \left|\{i \in \{1, \dotsc, 6\}:\ \alpha\(i\) = e_1 \}\right|.$$
We must show $K_{n,m}\(r\)<0$ for some $r>0$. Let $ a_1, \dots, a_8 $ be the coefficients of the polynomial $ P_{n,m} $ in the order stated in (\ref{formula laplacian cubed}).
Since $ Z_{i,j}\(0,\alpha\) = 4 $ if $ \alpha\(i\) \neq \alpha\(j\) $ and $ Z_{i,j}\(0,\alpha\) = 0 $ otherwise, due to symmetry, we note $ K_{n,m}\(r\) $ is a quintic polynomial of the form 
$$ K_{n,m}\(r\) = 64r\(b_0 + b_1 r + b_2 r^2 + b_1 r^3 + b_0 r^4\),$$ 
where
\begin{align*}
b_0 &:= 2a_1 + a_2 + a_4,\allowdisplaybreaks\\
b_1 &:= 8a_1 + 4a_2 + 4a_3 + 2a_4 + 2a_6 + 2a_7,\allowdisplaybreaks\\
b_2 &:= 12a_1 + 6a_2 + 8a_3 + 2a_4 + 4a_6 + 4a_7 + 8a_8. 
\end{align*}
For $ r > 0 $, we use the standard method of reducing the degree of a palindromic polynomial by considering 
\begin{align*}
r^{-3}K_{n,m}\(r\)&=64\(b_0\(r^2 + r^{-2}\) + b_1\(r + r^{-1}\) + b_2\)\\
&=64\(b_0\(r + r^{-1}\)^2 + b_1\(r + r^{-1}\) + b_2 - 2b_0\).
\end{align*}
It therefore suffices to analyze whether this quadratic in $r+r^{-1}$ attains a negative value in $ \[2, \infty\)$, which is the case if
\begin{equation}\label{Th2Eq3}
b_1^2 - 4b_0\(b_2 - 2b_0\)>0\quad\text{and}\quad-\frac{b_1}{2b_0}\ge2,\quad\text{i.e.}\quad b_1+4b_0\le0
\end{equation}
since 
$$b_0=8\(m-1\)\(m-2\)\(n-2m+2\)\(n-2m+4\)>0.$$
We compute
\begin{align*}
b_1+4b_0&=-8\(n-2m+2\)\big(n^3-\(12m-13\)n^2+\(39m^2-79m+38\)n\\
&\quad-2\(m-1\)\(22m^2-37m+4\)\big)
\end{align*}
and
$$b_1^2 - 4b_0\(b_2 - 2b_0\)=64\(n-2m+2\)\(n-3m+3\)\(n-4m+6\)\gamma_m\(n\),$$
where $\gamma_m$ is as in \eqref{Th2Eq1}. We let $N_m$ be the largest real root of $\gamma_m$. We compute 
$$\gamma_m\(4m-6\)=-96\(m-1\)^2\(m-2\)\(m-3\)^2<0.$$
Since the leading coefficient of $\gamma_m$ is positive and $2m-2<3m-3<4m-6$ for $m\ge4$, we then obtain {$N_m>4m-6$ and} $b_1^2 - 4b_0\(b_2 - 2b_0\)>0$ for all $n>N_m$. To show $b_1+4b_0\le0$, we consider the largest real root $N_m^\prime$ of
$$n^3-\(12m-13\)n^2+\(39m^2-79m+38\)n-2\(m-1\)\(22m^2-37m+4\)$$
as a polynomial in $n$.  A simple Euclidean division gives
\begin{align*}
\gamma_m\(n\)&=\(n^2-\(7m-6\)n-6\(m-1\)\(m-4\)\)\(n^3-\(12m-13\)n^2\right.\\
&\qquad\left.+\(39m^2-79m+38\)n-2\(m-1\)\(22m^2-37m+4\)\)\allowdisplaybreaks\\
&\quad-4\(m-1\)\(m-2\)\(\(23m-27\)n^2-2\(51m^2-106m+57\)n\right.\\
&\qquad\left.+4\(34m^3-89m^2+63m-6\)\).
\end{align*}
Moreover, the remainder polynomial in this division is negative for all $n\in\N$, which can be seen by noting $23m-27>0$ and
\begin{align*}
&4\(51m^2-106m+57\)^2-16\(23m-27\)\(34m^3-89m^2+63m-6\)\\
&\quad=-4\(527m^4-1048m^3-1642m^2+4728m-2601\)\allowdisplaybreaks\\
&\quad=-4\big(527\(m-2\)^4+3168\(m-2\)^3+4718\(m-2\)^2+2448\(m-2\)\\
&\qquad+335\big)\\
&\quad<0.
\end{align*}
Therefore, we obtain $\gamma_m\(N_m^\prime\)<0$, which implies $N_m^\prime<N_m$, and therefore if $n>N_m$, then both inequalities in \eqref{Th2Eq3} are satisfied, which in turn implies $K_{n,m}\(r\)<0$ for some $r>0$. Applying \eqref{Th2Eq2} then gives $Q_{g_{r,\varepsilon}}^{\(6\)}\(0\)<0$ provided $\varepsilon$ is sufficiently small. 

\smallskip
We now show $N_m\sim\Lambda m$ as $m\to\infty$ and $N_m\le\Lambda m$ for all $m\in\N$ with $m \geq 4$, where $\Lambda\simeq10.55$ is the largest real root of the polynomial $n^4-17n^3+83n^2-171n+140$. We let $\tilde{n}:=n/m$. We then write
\begin{align*}
\gamma_m\(n\)&=\(\tilde{n}-2\)\(\tilde{n}^4-17\tilde{n}^3+83\tilde{n}^2-171\tilde{n}+140\)m^5+\(19\tilde{n}^4-212\tilde{n}^3\right.\\
&\quad\ \left.+851\tilde{n}^2-1518\tilde{n}+1028\)m^4+4\(23\tilde{n}^3-163\tilde{n}^2+391\tilde{n}-320\)m^3\\
&\quad+4\(35\tilde{n}^2-144\tilde{n}+157\)m^2+48\(\tilde{n}-2\)m.
\end{align*}
In particular, considering the leading term, we obtain $N_m\sim\Lambda m$ as $m\to\infty$. Moreover, after analyzing the polynomials in $\tilde{n}$ appearing in the other terms, it turns out that all of them are positive for all $\tilde{n}\ge\Lambda$, which implies {$N_m<\Lambda m$} for all $m\in\N$ with $m \geq 4$.

\smallskip
Finally, we remark that for every $r,\varepsilon>0$, by stereographic projection, $g_{r,\varepsilon}$ extends to a smooth metric on $\S^n$. Indeed, if we let $\Phi_{x_0}:\S^n\setminus\left\{x_0\right\}\to\R^n$ be the stereographic projection from some pole $x_0\in\S^n$, then a standard computation using conformal invariance gives
$$u_{r,\varepsilon}\circ\Phi_{x_0}=\(\frac{2}{1+\left|\Phi_{x_0}\right|^2}\)^{\frac{n-2m}{2}}\int_{\S^n}G_{g_0}^{\(2m\)}\(\cdot,y\)\widehat{h}_{r,\varepsilon}\(y\)dv_{g_0}\(y\)$$
and
$$\(\Phi_{x_0}\)^{\ast}g_{r,\varepsilon}=\(\int_{\S^n}G_{g_0}^{\(2m\)}\(\cdot,y\)\widehat{h}_{r,\varepsilon}\(y\)dv_{g_0}\(y\)\)^{\frac{4}{n-2m}}g_0,$$
where $g_0$ is the standard metric on $\S^n$, $G_{g_0}^{\(2m\)}$ is the Green's function of the GJMS operator of order $2m$ on $\(\S^n,g_0\)$, $dv_{g_0}$ is the volume element of $\(\S^n,g_0\)$, and $\widehat{h}_{r,\varepsilon}$ is the smooth function on $\S^n$ defined by
\begin{align*}
\widehat{h}_{r,\varepsilon}&:=\(\frac{1+\left|\Phi_{x_0}\right|^2}{2}\)^{\frac{n+2m}{2}}h_{r,\varepsilon}\circ\Phi_{x_0}\\
&=2^{-\frac{n+2m}{2}}\big(\varepsilon +\big(1+\left|\Phi_{x_0}\right|^2\big)^{\frac{n+2m}{2}}\(r \rho_\varepsilon\(\Phi_{x_0}\(\cdot\) -e_1\) + \rho_\varepsilon\(\Phi_{x_0}\(\cdot\) + e_1\)\)\big).
\end{align*}
This ends the proof of Theorem~\ref{Th2}.
\endproof


\begin{thebibliography}{99}

\bib{AKRW}{article}{
   author={Andrade, J. H.},
   author={K\"onig, T.},  
   author={Ratzkin, J.},
   author={Wei, J.},     
   title={Quantitative stability of the total $Q$-curvature near minimizing metrics},
   note={arXiv:2407.06934},
   date={2024},
}

\bib{B}{article}{
   author={Branson, T. P.},
   title={Sharp inequalities, the functional determinant, and the complementary series},
   journal={Trans. Amer. Math. Soc.},
   volume={347},
   date={1995},
   number={10},
   pages={3671--3742},
}

\bib{BO}{article}{
   author={Branson, T. P.},
   author={\O rsted, B.},
   title={Explicit functional determinants in four dimensions},
   journal={Proc. Amer. Math. Soc.},
   volume={113},
   date={1991},
   number={3},
   pages={669--682},
}

\bib{CGS}{article}{
   author={Caffarelli, L. A.},
   author={Gidas, B.},
   author={Spruck, J.},
   title={Asymptotic symmetry and local behavior of semilinear elliptic equations with critical Sobolev growth},
   journal={Comm. Pure Appl. Math.},
   volume={42},
   date={1989},
   number={3},
   pages={271--297},
}

\bib{C}{article}{
   author={Case, J. S.},
   title={The Obata--V\'etois argument and its applications},
   journal={J. Reine Angew. Math.},
   volume={815},
   date={2024},
   pages={23--40},
}

\bib{CG}{article}{
   author={Case, J. S.},
   author={Gover, A. R.},
   title={The GJMS operators in geometry, analysis and physics},
   journal={J. Lond. Math. Soc. (2)},
   volume={113},
   date={2026},
   number={1, e70375},
}

\bib{CQY}{article}{
   author={Chang, S.-Y. A.},
   author={Qing, J.},
   author={Yang, P. C.},
   title={On the Chern--Gauss--Bonnet integral for conformal metrics on $\R^4$},
   journal={Duke Math. J.},
   volume={103},
   date={2000},
   number={3},
   pages={523--544},
}

\bib{CH}{article}{
   author={Chen, X.},
   author={Hou, F.},
   title={Remarks on GJMS operator of order six},
   journal={Pacific J. Math.},
   volume={289},
   date={2017},
   number={1},
   pages={35--70},
}

\bib{DHL}{article}{
   author={Djadli, Z.},
   author={Hebey, E.},
   author={Ledoux, M.},
   title={Paneitz-type operators and applications},
   journal={Duke Math. J.},
   volume={104},
   date={2000},
   number={1},
   pages={129--169},
}

\bib{ER}{article}{
   author={Esposito, P.},
   author={Robert, F.},
   title={Mountain pass critical points for Paneitz--Branson operators},
   journal={Calc. Var. Partial Differential Equations},
   volume={15},
   date={2002},
   number={4},
   pages={493--517},
}

\bib{Fa}{article}{
   author={Fang, H.},
   title={On a conformal Gauss--Bonnet--Chern inequality for LCF manifolds and related topics},
   journal={Calc. Var. Partial Differential Equations},
   volume={23},
   date={2005},
   number={4},
   pages={469--496},
}

\bib{FWX}{article}{
   author={Fazly, M.},
   author={Wei, J.},
   author={Xu, X.},
   title={A pointwise inequality for the fourth-order Lane--Emden equation},
   journal={Anal. PDE},
   volume={8},
   date={2015},
   number={7},
   pages={1541--1563},
}

\bib{FG1}{article}{
   author={Fefferman, C.},
   author={Graham, C. R.},
   title={Conformal invariants},
   journal={\'Elie Cartan et les math\'emati\-ques d'aujourd'hui -- Lyon, 25--29 juin 1984, Ast\'{e}risque},
   date={1985},
   number={S131},
   pages={95--116},
}

\bib{FG2}{book}{
   author={Fefferman, C.},
   author={Graham, C. R.},
   title={The ambient metric},
   series={Annals of Mathematics Studies},
   volume={178},
   publisher={Princeton University Press, Princeton, NJ},
   date={2012},
}

\bib{Fi}{article}{
   author={Finn, R.},
   title={On a class of conformal metrics, with application to differential geometry in the large},
   journal={Comment. Math. Helv.},
   volume={40},
   date={1965},
   pages={1--30},
}

\bib{GWW}{article}{
   author={Ge, Y.},
   author={Wang, G.},  
   author={Wei, W.},  
   title={A Yamabe problem for the quotient between the $Q$-curvature and the scalar curvature},
   note={arXiv:2603.15074},
   date={2026},
}

\bib{GKW}{article}{
   author={Gong, M.},
   author={Kim, S.},  
   author={Wei, J.},  
   title={Compactness and non-compactness theorems of the fourth-order and sixth-order constant $Q$-curvature problems},
   note={arXiv:2502.14237},
   date={2025},
}

\bib{GJMS}{article}{
   author={Graham, C. R.},
   author={Jenne, R.},
   author={Mason, L. J.},
   author={Sparling, G. A. J.},
   title={Conformally invariant powers of the Laplacian. I. Existence},
   journal={J. London Math. Soc. (2)},
   volume={46},
   date={1992},
   number={3},
   pages={557--565},
}

\bib{GHL}{article}{
   author={Gursky, M. J.},
   author={Hang, F.},
   author={Lin, Y.-J.},
   title={Riemannian manifolds with positive Yamabe invariant and Paneitz operator},
   journal={Int. Math. Res. Not. IMRN},
   date={2016},
   number={5},
   pages={1348--1367},
}

\bib{GM}{article}{
   author={Gursky, M. J.},
   author={Malchiodi, A.},
   title={A strong maximum principle for the Paneitz operator and a non-local flow for the $Q$-curvature},
   journal={J. Eur. Math. Soc. (JEMS)},
   volume={17},
   date={2015},
   number={9},
   pages={2137--2173},
}

\bib{HY1}{article}{
   author={Hang, F.},
   author={Yang, P. C.},
   title={Sign of Green's function of Paneitz operators and the $Q$ curvature},
   journal={Int. Math. Res. Not. IMRN},
   date={2015},
   number={19},
   pages={9775--9791},
}

\bib{HY2}{article}{
   author={Hang, F.},
   author={Yang, P. C.},
   title={$Q$-curvature on a class of manifolds with dimension at least 5},
   journal={Comm. Pure Appl. Math.},
   volume={69},
   date={2016},
   number={8},
   pages={1452--1491},
}

\bib{HR}{article}{
   author={Hebey, E.},
   author={Robert, F.},
   title={Compactness and global estimates for the geometric Paneitz equation in high dimensions},
   journal={Electron. Res. Announc. Amer. Math. Soc.},
   volume={10},
   date={2004},
   pages={135--141},
}

\bib{HRW}{article}{
   author={Hebey, E.},
   author={Robert, F.},
   author={Wen, Y.},
   title={Compactness and global estimates for a fourth order equation of critical Sobolev growth arising from conformal geometry},
   journal={Commun. Contemp. Math.},
   volume={8},
   date={2006},
   number={1},
   pages={9--65},
}

\bib{HS}{article}{
   author={Hyder, A.},
   author={Sire, Y,},
   title={Singular solutions for the constant $Q$-curvature problem},
   journal={J. Funct. Anal.},
   volume={280},
   date={2021},
   number={3},
   pages={Paper No. 108819, 39},
}

\bib{Li1}{article}{
   author={Li, M.},
   title={The total $Q$-curvature, volume entropy and polynomial growth polyharmonic functions},
   journal={Adv. Math.},
   volume={450},
   number={109768},
   date={2024},
}

\bib{Li2}{article}{
   author={Li, M.},
   title={The total $Q$-curvature, volume entropy and polynomial growth polyharmonic functions (II)},
   note={arXiv:2408.03640},
   date={2024},
}

\bib{Li3}{article}{
   author={Li, M.},
   title={Conformal metrics with finite total $Q$-curvature revisited},
   note={arXiv:2405.\break09872},
   date={2024},
}

\bib{LW}{article}{
   author={Li, M.},
   author={Wei, J.},
   title={A remark on the Case--Gursky--V\'etois identity and its applications},
   journal={Proc. Amer. Math. Soc.},
   volume={153},
   date={2025},
   number={8},
   pages={3417--3430},
}

\bib{LWX}{article}{
   author={Li, M.},
   author={Wei, J.},   
   author={Xu, X.},
   title={On geometry of $Q_g^{\(2k\)}$-curvature},
   note={arXiv:2506.20165},
   date={2025},
}

\bib{LX1}{article}{
   author={Li, Y.},
   author={Xiong, J.},
   title={Compactness of conformal metrics with constant $Q$-curvature. I},
   journal={Adv. Math.},
   volume={345},
   date={2019},
   pages={116--160},
}

\bib{LX2}{article}{
   author={Li, M.},
   author={Xu, X.},
   title={On positivity of the $Q$-curvatures of conformal metrics},
   journal={J. Funct. Anal.},
   volume={289},
   date={2025},
   number={8, 111011},
}

\bib{LX3}{article}{
   author={Li, M.},   
   author={Xu, X.},
   title={A sharp isoperimetric inequality and the top order $Q$-curvature},
   note={arXiv:2607.06951},
   date={2025},
}

\bib{Lin}{article}{
   author={Lin, C.-S.},
   title={A classification of solutions of a conformally invariant fourth order equation in ${\bf R}^n$},
   journal={Comment. Math. Helv.},
   volume={73},
   date={1998},
   number={2},
   pages={206--231},
}

\bib{M}{article}{
   author={Martinazzi, L.},
   title={Classification of solutions to the higher order Liouville's equation on $\R^{2m}$},
   journal={Math. Z.},
   volume={263},
   date={2009},
   number={2},
   pages={307--329},
}

\bib{MP}{article}{
   author={Mazumdar, S.},
   author={Premoselli, B.},   
   title={Compactness of conformal metrics with constant $Q$-curvature of higher order},
   note={arXiv:2603.06249},
   date={2026},
}

\bib{MN}{article}{
   author={Mazumdar, S.},
   author={Ndiaye, C. B.},   
   title={Barycenter technique for the higher order $Q$-curvature equation},
   note={arXiv:2510.00888},
   date={2025},
}

\bib{MV}{article}{
   author={Mazumdar, S.},
   author={V\'etois, J.},
   title={Existence results for the higher-order $Q$-curvature equation},
   journal={Calc. Var. Partial Differential Equations},
   volume={63},
   date={2024},
   number={6},
   pages={Paper No. 151, 29},
}

\bib{NX}{article}{
   author={Ndiaye, C. B.},
   author={Xiao, J.},
   title={An upper bound of the total $Q$-curvature and its isoperimetric deficit for higher-dimensional conformal Euclidean metrics},
   journal={Calc. Var. Partial Differential Equations},
   volume={38},
  date={2010},
   number={1-2},
   pages={1--27},
}

\bib{O}{article}{
   author={Obata, M.},
   title={The conjectures on conformal transformations of Riemannian manifolds},
   journal={J. Differential Geometry},
   volume={6},
   date={1971/72},
}

\bib{R1}{article}{
   author={Robert, F.},
   title={Positive solutions for a fourth order equation invariant under isometries},
   journal={Proc. Amer. Math. Soc.},
   volume={131},
   date={2003},
   number={5},
   pages={1423--1431},
}

\bib{R2}{article}{
   author={Robert, F.},
   title={Admissible $Q$-curvatures under isometries for the conformal GJMS operators},
   conference={
      title={Nonlinear elliptic partial differential equations},
   },
   book={
      series={Contemp. Math.},
      volume={540},
      publisher={Amer. Math. Soc., Providence, RI},
   },
   date={2011},
   pages={241--259},
}

\bib{V}{article}{
   author={V\'etois, J.},
   title={Uniqueness of conformal metrics with constant $Q$-curvature on closed Einstein manifolds},
   journal={Potential Anal.},
   volume={61},
   date={2024},
   number={3},
   pages={485--500},
}

\bib{WX}{article}{
   author={Wei, J.},
   author={Xu, X.},
   title={Classification of solutions of higher order conformally invariant equations},
   journal={Math. Ann.},
   volume={313},
   date={1999},
   number={2},
   pages={207--228},
}

\end{thebibliography}
\end{document}